\documentclass[12pt]{article}
\usepackage{hyperref}
\usepackage{amsthm}
\usepackage{amssymb}
\usepackage{array}
\usepackage{latexsym}
\usepackage{enumerate}
\usepackage{amsmath}
\usepackage{amsfonts}
\usepackage{graphicx}
\usepackage{psfrag}
\usepackage{comment}
\usepackage{breqn}
\usepackage[english]{babel}
\usepackage[T1]{fontenc}
\usepackage{geometry}
\newcommand{\rank}{\mathrm{rank}\,}
\ifx\smallsetminus\undefined
\def\smallsetminus{\setminus}
\fi
\title{Intersections of the  Hermitian surface with irreducible
  quadrics in even characteristic}
\author{A.~Aguglia\footnote{Dipartimento di Meccanica, Matematica e Management, Politecnico di Bari,
  Via Orabona 4, I-70126 Bari, Italy}
,\,  \,L.~Giuzzi\footnote{DICATAM, Section of Mathematics, Universit\`a di
Brescia, Via Branze 43, I-25123 Brescia, Italy}}
\date{}
\theoremstyle{plain}
\newtheorem{prop}{Proposition}[section]
\newtheorem{theorem}[prop]{Theorem}

\newtheorem{lemma}[prop]{Lemma}
\theoremstyle{definition}

\newcommand\cR{{\mathcal R}}

\newcommand\cH{\mathcal H}
\newcommand{\cQ}{\mathcal Q}

\newcommand{\PGU}{\mathrm{PGU}}

\newcommand{\PG}{\mathrm{PG}}
\newcommand{\AG}{\mathrm{AG}}
\newcommand{\GF}{\mathrm{GF}}

\newcommand{\Tr}{\mathrm{Tr}}
\newcommand{\cC}{\mathcal C}

\begin{document}

\maketitle

\begin{abstract}
We determine the possible intersection sizes of a Hermitian surface $\cH$
with an irreducible quadric of $\PG(3,q^2)$
sharing at least a tangent plane at a common non-singular point
when $q$ is even.
\end{abstract}

\section{Introduction}
The study of intersections of geometric objects is a classical problem
in geometry; see e.g. \cite{EH,fulton}.
In the case of combinatorial geometry, it has
several possible applications either
to characterize configurations
or to construct new codes.
\par
Let $\cC$ be a projective $[n,k]$ linear code over $\GF(q)$.
It is always possible to consider the set of points
$\Omega$ in $\PG(k-1,q)$ whose coordinates correspond to the
columns of any generating matrix for $\cC$. Under this setup
the problem of determining the minimum weight of $\cC$ can be
reinterpreted, in a purely geometric setting, as finding
the largest hyperplane sections of $\Omega$.
More in detail, any codeword $c\in\cC$ corresponds to a linear functional
evaluated on the points of $\Omega$; see \cite{L,tvn}.
For examples of applications of these techniques see \cite{IG1,IG2,IG3}.

Clearly, it is not necessary to restrict the study to hyperplanes.
The higher weights of $\cC$ correspond to sections of $\cC$
with subspaces of codimension larger than $1$; see \cite{tv} and
also \cite{HTV} for  Hermitian varieties.

A different generalization consists in studying codes arising
from the evaluation
on $\Omega$ of functionals of degree $t>1$; see \cite{L}.
These constructions  yield, once more, linear codes, whose
weight distributions depend on the intersection patterns of $\Omega$ with
all possible algebraic hypersurfaces of $\PG(k-1,q)$ of degree $t$.

The case of quadratic functional codes on Hermitian varieties has been
extensively investigated in recent years; see \cite{BBFS,E1,EHRS,E2,E3,HS}.
However, it is still an open problem to classify all possible intersection
numbers and patterns
between a quadric surface $\cQ$ in $\PG(3,q^2)$ and
a Hermitian surface $\cH=\cH(3,q^2)$.

In \cite{AG}, we determined the possible intersection
numbers between $\cQ$ and  $\cH$ in $\PG(3,q^2)$ under
the assumption that $q$ is an odd prime power and $\cQ$ and $\cH$
share at least one tangent plane.
The same problem has been studied independently also in \cite{CP} for
$\cQ$ an elliptic quadric; this latter work contains
also  some results for $q$ even.

In this paper we fully extend the arguments of \cite{AG} to the case of
$q$ even. It turns out that the geometric properties being considered as
well as the algebraic conditions to impose are
different and more involved than those for the odd $q$ case.
Our main result is contained in the following theorem.
\begin{theorem}\label{main1}
In $\PG(3,q^2)$, with $q$ even, let $\cH$ and $\cQ$
be respectively a Hermitian surface and
an irreducible quadric with the same tangent plane at at least one common
non-singular point $P$.
Then, the possible sizes of the intersection $\cH\cap\cQ$ are
as follows.
\begin{itemize}
\item For $\cQ$ elliptic:
\[q^3-q^2+1,q^3-q^2+q+1, q^3-q+1, q^3+1, q^3+q+1, q^3+q^2-q+1, q^3 + q^2+1. \]
\item For $\cQ$ a quadratic cone:
\[ q^3-q^2+q+1,  q^3-q+1, q^3+q+1, q^3+q^2-q+1,   q^3 + 2q^2-q+1.\]
\item For $\cQ$ hyperbolic:
\[ q^2+1,q^3-q^2+1,  q^3-q^2+q+1, q^3-q+1, q^3+1,  q^3+q+1,
q^3+q^2-q+1, q^3+ q^2+1, \]
\[   q^3+2q^2-q+1,  q^3+2q^2+1,  q^3+3q^2-q+1,   2q^3+q^2+1.  \]
\end{itemize}
\end{theorem}
We remark that, as we are dealing with irreducible quadrics in $PG(3,q^2)$, by quadratic cone
(or, in short, cones) we shall always mean in dimension $3$ the quadric projecting an irreducible
conic contained in a plane $\pi$ from a vertex $V\not\in\pi$.

Our methods are algebraic in nature, based upon
the $\GF(q)$--linear representation of vector spaces over $\GF(q^2)$,
but in order to  rule out some cases a more geometric and combinatorial
approach is needed as well as some considerations on the action of the
unitary groups.

For generalities on Hermitian varieties in projective spaces the reader is referred to \cite{PGOFF, S}.
Basic notions on quadrics over finite fields are found in \cite{HT,PGOFF}.

\section{Invariants of quadrics}
\label{s2}

In this section we recall some basic invariants of quadrics in
even characteristic; the main reference for these results is \cite[\S 1.1, 1.2]{HT},
whose notation and approach we closely follow.

Recall that a quadric $\cQ$ in $\PG(n,q)$ is just the set of points
$(x_0,\ldots,x_n)\in\PG(n,q)$ such that $F(x_0,\ldots,x_n)=0$ for
some non-null quadratic form
\[ F(x_0,\ldots,x_n)=\sum_{i=0}^n a_ix_i^2+\sum_{i<j} a_{ij}x_ix_j. \]
If there is no change of coordinates reducing $F$ to a form in
fewer variables, then $\cQ$ is \emph{non-degenerate}; otherwise
$\cQ$ is \emph{degenerate}. The minimum number of indeterminates
which may appear in an equation for $\cQ$ is the \emph{rank} of
the quadric, denoted by $\rank\cQ$; see \cite[\S 15.3]{FPS3D}.

Let consider the  quadric $\cQ$ in $\PG(3,q)$  of equation $\sum_{i=0}^3 a_ix_i^2+\sum_{i<j} a_{ij}x_ix_j=0$ and
define
\[ A:=\begin{pmatrix}
   2a_0 & a_{01} & a_{02} & a_{03} \\
   a_{01} & 2a_{1} & a_{12} & a_{13} \\
   a_{02} & a_{12} & 2a_{2} & a_{23} \\
   a_{03} & a_{13} & a_{23} & 2a_{3}
   \end{pmatrix},\quad
   B:=\begin{pmatrix}
   0  & a_{01} & a_{02} & a_{03} \\
   -a_{01} & 0 & a_{12} & a_{13} \\
   -a_{02} & -a_{12} & 0 & a_{23} \\
   -a_{03} & -a_{13} & -a_{23} & 0
   \end{pmatrix} \]
and
\begin{equation}\label{eqAlpha} \alpha:=
  \frac{\det A-\det B}{4\det B}.
\end{equation}

When $q$ is even
 $\det A$, $\det B$ and $\alpha$ should be interpreted as follows. In $A$ and $B$ we replace the terms
 $a_i$ and  $a_{ij}$ by indeterminates  $Z_i$ and $Z_{ij}$ and we evaluate  $\det A$, $\det B$ and $\alpha$
 as rational functions over the integer ring $\mathbb{Z}$. Then we specialize $Z_i$ and  $Z_{i}$ to $a_i$ and $a_{i,j}$.

Actually, it turns out that $\cQ$ is non-degenerate if and only
if $\det A\neq 0$; see \cite[Theorem 1.2]{HT}.
By the same theorem, when $q=2^h$, a non-degenerate quadric
$\cQ$ of $PG(3,q)$ is hyperbolic or elliptic
according as
\[ \Tr_{q}(\alpha)=0 \mbox{ or } \Tr_{q}(\alpha)=1, \]
respectively, where
$\Tr_{q}$ denotes the absolute trace $\GF(q)\to \GF(2)$ which maps $x\in \GF(q)$
to $x+x^2+x^{2^2}+\ldots+x^{2^{h-1}}$.

\section{Some technical tools}
In this section we are going to prove a series of  lemmas that we will be useful in the proof of our main result, namely  Theorem \ref{main1}.

Henceforth, we shall assume $q$ to be even; $x,y,z$ will denote affine coordinates
in $\AG(3,q^2)$ and the corresponding homogeneous coordinates will be
$J,X,Y,Z$. The hyperplane at infinity
of $\AG(3,q^2)$, denoted as $\Sigma_{\infty}$,  has equation
$J=0$.

Since all non-degenerate Hermitian surfaces of $\PG(3,q^2)$ are projectively
equivalent,
we can assume, without loss of generality,
$\cH$ to have affine equation
\begin{equation}\label{eqH}
z^q+z=x^{q+1}+y^{q+1}.
\end{equation}
Since $\PGU(4,q)$ is transitive on $\cH$, see \cite[\S 35]{S}, we can also suppose that a point
with common tangent plane to  $\cH$ and $\cQ$ is
$P=P_{\infty}(0,0,0,1) \in\cH$; under these assumptions, the tangent plane at $P$ to $\cH$
is $\Sigma_{\infty}$.
Under the aforementioned assumptions,   $\cQ$
has affine equation
\begin{equation}\label{eqQ}
z=ax^{2}+by^{2}+cxy+dx+ey+f
\end{equation}
with $a,b,c,d,e,f\in\GF(q^2)$.
A direct computation proves that $\cQ$ is non-degenerate if and only if $c\neq0$; furthermore
$\cQ$ is hyperbolic or elliptic according as the value of
$$\Tr_{q^2}(ab/c^2)$$ is respectively $0$ or $1$.
When $c=0$ and $(a,b)\neq (0,0)$, the quadric $\cQ$ is a cone with vertex
a single point.

Write now $\cC_{\infty}:=\cQ\cap\cH\cap\Sigma_{\infty}$.
If $\cQ$ is elliptic,
the point $P_{\infty}$ is, clearly, the only point at infinity of
$\cQ\cap\cH$, that is
$\cC_{\infty}=\{P_{\infty}\}$.
The nature of
$\cC_{\infty}$ when $\cQ$ is either hyperbolic or a  cone, is
detailed by the following lemma.
\begin{lemma}
\label{l0}
 If $\cQ$ is a  cone, then $\cC_{\infty}$
  consists of either $1$ point or $q^2+1$ points on a line.
  When $\cQ$ is a hyperbolic quadric, then $\cC_{\infty}$ consists of
  either $1$ point, or $q^2+1$ points on a line or $2q^2+1$ points on two lines. All cases may
  actually occur.
\end{lemma}
\begin{proof}
As both $\cH\cap\Sigma_{\infty}$ and $\cQ\cap\Sigma_{\infty}$ split in
lines through $P_{\infty}$, it is straightforward to
see that the only possibilities for $\cC_{\infty}$ are
those outlined above; in particular, when $\cQ$ is hyperbolic,
 $\cC_{\infty}$  consists of either $1$ point or $1$ or $2$ lines.
It is straightforward to see that all cases may actually occur, as given
any two lines $\ell,m$ in $\PG(3,q^2)$ there always exist
at least one hyperbolic quadric containing both $m$ and $\ell$.
Likewise, given a line $\ell\in\Sigma_{\infty}$ with $P\in\ell$
there always is at least one cone with vertex $V\in\ell$ and $V\neq P$
meeting $\Sigma_{\infty}$ just in $\ell$.
\end{proof}
We now determine the number of affine points that $\cQ$ and $\cH$ have in common.
\begin{lemma}\label{main2}
The possible sizes of $(\cH\cap\cQ)\setminus\Sigma_{\infty}$,
are either
\[q^3-q^2,q^3-q^2+q, q^3-q, q^3, q^3+ q, q^3+q^2-q,q^3+q^2\]
when $\cQ$ is elliptic,
or
\[q^3-q^2+q, q^3-q, q^3+ q, q^3+q^2-q,\]
when $\cQ$ is a  cone,
or
\[ q^2, q^3-q^2, q^3-q^2+q, q^3-q, q^3, q^3+ q, q^3+q^2-q,q^3+q^2, 2q^3-q^2\]
 when $\cQ$ is hyperbolic.
\end{lemma}
\begin{proof}
We have to study the following system of equations
\begin{dmath}
  \label{sis1}
  \begin{cases}
    z^q+z=x^{q+1}+y^{q+1}     \\
       z=ax^{2}+by^{2}+cxy+dx+ey+f.
    \end{cases}
  \end{dmath}
In order to
solve \eqref{sis1}, recover the value of $z$ from the second
equation and substitute it in the first. This gives
\begin{dmath}
\label{eqc}
a^qx^{2q}+b^qy^{2q}+c^qx^qy^q+d^qx^q+e^qy^q+f^q
+ax^{2}+by^{2}\\+cxy+dx+ey+f
=x^{q+1}+y^{q+1}.
\end{dmath}
Consider $\GF(q^2)$ as a vector space over $\GF(q)$ and fix a
basis $\{1,\varepsilon\}$ with $\varepsilon\in\GF(q^2)\setminus
\GF(q)$. Write any element in $\GF(q^2)$ as a linear
combination with respect to this basis, that is, for any $x\in\GF(q^2)$
let $x=x_0+x_1\varepsilon$, where
$x_0,x_1\in\GF(q)$. Thus, \eqref{eqc} can be studied as an equation over
$\GF(q)$ in the indeterminates $x_0,x_1,y_0,y_1$.
Analogously write also $a=a_0+\varepsilon a_1$,
$b=b_0+\varepsilon b_1$ and so on.

As $q$ is even, it is always possible to choose
$\varepsilon \in \GF(q^2) \setminus\GF(q)$
such that $\varepsilon^2+\varepsilon+\nu=0$, for some
$\nu \in \GF(q)\setminus \{1\}$ and $\Tr_{q}(\nu)=1$.
Then, also, $\varepsilon^{2q}+\varepsilon^q+\nu=0$. Therefore,
$(\varepsilon^q+\varepsilon)^2+(\varepsilon^q+\varepsilon)=0$,
whence $\varepsilon^q+\varepsilon+1=0$. With
this choice of $\varepsilon$, \eqref{eqc} reads as
\begin{dmath}
\label{eqeven1}
(a_1+1)x_0^2+x_0x_1+[a_0+(1+\nu)a_1+\nu]x_1^2+(b_1+1)y_0^2+y_0y_1
+[b_0+(1+\nu)b_1+\nu]y_1^2+c_1x_0y_0+(c_0+c_1)x_0y_1+(c_0+c_1)x_1y_0
+[c_0+(1+\nu)c_1]x_1y_1
+d_1x_0+(d_0+d_1)x_1+e_1y_0+(e_0+e_1)y_1+f_1=0.
\end{dmath}
The solutions $(x_0,x_1,y_0,y_1)$ of
\eqref{eqeven1} correspond to the affine points of a (possibly degenerate)
 quadratic hypersurface $\Xi$ of $\PG(4,q)$.
Recall that
the number $N$ of affine points of $\Xi$ equals
the number of points of $\cH\cap\cQ$ which lie in
$\AG(3,q^2)$; we shall use the results of
\cite{HT} in order to actually count these points.

To this purpose, we  first  determine the number of  points
at infinity of $\Xi$.
These points are those of the quadric $\Xi_\infty$ of
$\PG(3,q)$ with equation
\begin{dmath}\label{eq00}
f(x_0,x_1,y_0,y_1)=(a_1+1)x_0^2+x_0x_1+[a_0+(1+\nu)a_1+\nu]x_1^2+(b_1+1)y_0^2+y_0y_1
+[b_0+(1+\nu)b_1+\nu]y_1^2+c_1x_0y_0+(c_0+c_1)x_0y_1+(c_0+c_1)x_1y_0
+[c_0+(1+\nu)c_1]x_1y_1=0.
\end{dmath}

Following the approach outlined in Section \ref{s2}, write
\[A_{\infty}=\begin{pmatrix}
    2(a_1+1) &1&c_1& c_0+c_1 \\
1 & 2[a_0+(1+\nu)a_1+\nu]  & c_0+c_1&c_0+(1+\nu)c_1  \\
   c_1 & c_0+c_1& 2(b_1+1) &1 & \\
    c_0+c_1&c_0+(1+\nu)c_1&1 & 2[b_0+(1+\nu)b_1+\nu]
   \end{pmatrix}. \]
As $q$ is even, a direct computation gives  $\det A_{\infty}=1+c^{2(q+1)}$. The quadric  $\Xi_{\infty}$ is  non-degenerate  if and only if
$\det A_{\infty}\neq 0$, that is $c^{q+1}\neq1$.

 When $\cQ$ is a cone,  namely $c=0$, it turns out that $\det A_{\infty}\neq 0$ and hence
 $\rank \Xi_{\infty}=4$.

Assume now $\cQ$ to be non-degenerate.
If the equation  of $\Xi_{\infty}$ were to be of the form
$f(x_0,x_1,y_0,y_1)=(lx_0+mx_1+ny_0+ry_1)^2$ with $l,m,n,r$ over some extension of $\GF(q)$,
then $c=0$; this is a contradiction. So $\rank \Xi_\infty\geq 2$.

 We are next going to show that when $\rank \Xi_{\infty}=2$, that is $\Xi_{\infty}$ splits into two planes,
 the quadric  $\cQ$  is hyperbolic, that is  $\Tr_{q^2}(ab/c^2)=0$.
 First observe that $c^{q+1}=1$ since the quadric $\Xi_{\infty}$ is degenerate.

Consider now the following $4$
intersections  $\cC_0:\Xi_{\infty}\cap [x_0=0] $, $\cC_1:\Xi_{\infty}\cap [x_1=0] $,
$\cC_2:\Xi_{\infty}\cap [y_0=0]$, $\cC_3:\Xi_{\infty}\cap [y_1=0]$.
Clearly, as $\Xi_{\infty}$ is, by assumption, reducible in the union of two planes,
all of these conics are  degenerate; thus
we get the following four formal equations
\[ \frac{1}{2}\det\begin{pmatrix}
 2[a_0+(1+\nu)a_1+\nu]  & c_0+c_1& c_0+(1+\nu)c_1  \\
    c_0+c_1& 2(b_1+1) &1 & \\
    c_0+(1+\nu)c_1&1 & 2[b_0+(1+\nu)b_1+\nu]
   \end{pmatrix}=0, \]
\[ \frac{1}{2}\det\begin{pmatrix}
    2(a_1+1) &c_1& c_0+c_1 \\
   c_1 & 2(b_1+1) &1 & \\
    c_0+c_1&1 & 2[b_0+(1+\nu)b_1+\nu]
   \end{pmatrix}=0, \]
\[ \frac{1}{2}\det\begin{pmatrix}
    2(a_1+1) &1& c_0+c_1 \\
1 & 2[a_0+(1+\nu)a_1+\nu]  &c_0+(1+\nu)c_1  \\
       c_0+c_1&[c_0+(1+\nu)c_1]& 2[b_0+(1+\nu)b_1+\nu]
   \end{pmatrix}=0,\]
\[\frac{1}{2}\det\begin{pmatrix}
    2(a_1+1) &1&c_1\\
1 & 2[a_0+(1+\nu)a_1+\nu]  & c_0+c_1  \\
   c_1 & c_0+c_1& 2(b_1+1)
   \end{pmatrix}=0. \]
Using the condition $c^{q+1}=1$, these give
\begin{equation}\label{sis2}
\begin{cases}
 a_0+(1+\nu)a_1+(c_0^2+c_1^2)b_0+\nu (c_0^2+c_1^2+c_1^2\nu)b_1=0\\
  a_1+c_1^2b_0+(c_0^2+\nu c_1^2)b_1=0\\
  (c_0^2+c_1^2)a_0+\nu(c_0^2+c_1^2+\nu c_1^2)a_1+b_0+(1+\nu)b_1=0\\
  c_1^2a_0+(c_0^2+\nu c_1^2)a_1+b_1=0.
\end{cases}
\end{equation}
If $c_1=0$, then $c_0=1$ since $c^{q+1}=1$ and,
 solving \eqref{sis2}, we obtain
 \begin{equation}\label{eq6}
 \begin{cases}
   a_1=b_1 \\
   a_0+a_1+b_0=0.
  \end{cases}
 \end{equation}
Therefore   $\Tr_{q^2}(ab/c^2)=\Tr_{q^2}((a_0+\varepsilon a_1)(a_0+(\varepsilon+1) a_1))=\Tr_{q^2}(a^{q+1})=0$ as $a^{q+1}\in\GF(q)$.

Suppose now  $c_1\neq 0$. System~\eqref{sis2} becomes
\begin{equation}
\label{eq7}
\begin{cases}
a_0=\left(\frac{c_0^2}{c_1^2}+\nu\right)a_1+\frac{b_1}{c_1^2}\\
b_0=\frac{a_1}{c_1^2}+\left(\frac{c_0^2}{c_1^2}+\nu\right)b_1;\\
\end{cases}
\end{equation}
 hence,
\[\frac{ab}{c^2}=\frac{(a_1^2+b_1^2)(c_1^2\nu+c_1^2\varepsilon+c_0^2)+a_1b_1(c_1^2\nu+c_1^2\varepsilon+c_0^2+1)^2}{c_1^4(c_0+\varepsilon c_1)^2}.\]
Since $ \varepsilon^2=\varepsilon +\nu $ and $c_1^2\nu+c_0^2+1=c_0c_1$, we get $$\frac{ab}{c^2}=\frac{a_1^2+b_1^2}{c_1^4}+\frac{a_1b_1}{c_1^2}\in\GF(q),$$
which gives $\Tr_{q^2}(ab/c^2)=0$ once more.
 Hence if $\rank \Xi_{\infty}=2$, then $\cQ$ is hyperbolic.

When $\rank \Xi_{\infty}=3$, then $\Xi_{\infty}$ is a cone comprising the join of a point to a conic.

Finally, when $\Xi_{\infty}$ is non-degenerate, we can establish the nature of $\Xi_{\infty}$ by computing
 the trace of $\alpha$ as given by~\eqref{eqAlpha}
 where
\[B=\begin{pmatrix}
    0 &1&c_1& c_0+c_1 \\
-1 & 0  & c_0+c_1&c_0+(1+\nu)c_1  \\
   -c_1 & -(c_0+c_1)& 0 &1 & \\
    -(c_0+c_1)&-c_0-(1+\nu)c_1&-1 & 0
   \end{pmatrix} .\]
Write  $\gamma=(1+c^{q+1})=(1+c_0^2+\nu c_1^2+c_0c_1)$. A straightforward
computation shows
\begin{multline}
\label{alfa}
\alpha=\frac{a_0+a_1+b_0+b_1}{\gamma}+
   \frac{1}{\gamma^2}\big[(1+\nu)(a_1^2+b_1^2)+\\ (c_0^2+c_1^2\nu)(a_0b_1+a_1b_0)+
a_0a_1+b_0b_1+a_0b_0c_1^2+a_1b_1c_0^2\big].
\end{multline}
Thus, the following possibilities for $N=|\Xi|-|\Xi_{\infty}|=|(\cH\cap\cQ)\cap \AG(3,q^2)|$  may occur according as:
\begin{enumerate}[(C1)]
\item\label{C1}
  $ \rank \Xi=5$ and  $ \rank \Xi_{\infty}=4$;
    \begin{enumerate}[\mbox{(C\ref{C1}.}1)]
\renewcommand{\theenumi}{\relax}
  \item\label{C1.1}  $\Xi$ is a parabolic quadric and  $\Xi_{\infty}$ is a hyperbolic quadric  ($\det A_{\infty}\neq 0$  and $\Tr_{q}(\alpha)=0$). Then,
\[N= (q+1)(q^2+1)-(q+1)^2=q^3-q. \]

 \item\label{C1.2}  $\Xi$ is a parabolic quadric and the quadric $\Xi_{\infty}$ is  elliptic ($\det A_{\infty}\neq 0$ and  $\Tr_{q}(\alpha)=1$). Then,
   \[N=  (q+1)(q^2+1)-(q^2+1)=q^3+q. \]
 \end{enumerate}

\item\label{C2} $\rank\Xi=5$ and $ \rank \Xi_{\infty}=3$ ($\det A_{\infty}=0$);

  $\Xi$ is a parabolic quadric and the hyperplane at infinity is tangent to $\Xi$; thus $\Xi_{\infty}$ is a cone
comprising the join of a point to a conic. Then,
 \[ N= (q+1)(q^2+1)-(q^2+q+1)=q^3. \]

\item\label{C3}
  $\rank \Xi=4$ and $ \rank \Xi_{\infty}=4$;
\begin{enumerate}[\mbox{(C\ref{C3}.}1)]
\renewcommand{\theenumi}{\relax}
  \item\label{C3.1} $\Xi$ is a cone  projecting a hyperbolic quadric of $\PG(3,q)$ and the quadric $\Xi_{\infty}$ is  hyperbolic ($\det A_{\infty}\neq 0$  and $\Tr_{q}(\alpha)=0$). Then,
 \[ N= q(q+1)^2+1-(q+1)^2=q^3+q^2-q. \]
\item\label{C3.2}  $\Xi$ is a cone  projecting an elliptic quadric of $\PG(3,q)$ and the quadric $\Xi_{\infty}$ is  elliptic ($\det A_{\infty}\neq 0$  and $\Tr_{q}(\alpha)=1$). Then,
\[ N= q(q^2+1)+1-(q^2+1)=q^3-q^2+q. \]
 \end{enumerate}

\item\label{C4} $\rank \Xi=4$, $\rank \Xi_{\infty}=3$;
\begin{enumerate}[\mbox{(C\ref{C4}.}1)]
\renewcommand{\theenumi}{\relax}
  \item\label{C4.1}$\Xi$ is a cone  projecting
a hyperbolic quadric and
 $\Xi_{\infty}$ is a cone comprising the join of a point to a conic. Then,
\[N=q(q+1)^2+1-[q(q+1)+1]=q^3+q^2.\]

\item \label{C4.2}
$\Xi$ is a cone  projecting
 an elliptic  quadric and $\Xi_{\infty}$ is a cone comprising the join of a point to a conic.
 Then,
\[N=q(q^2+1)+1-[q(q+1)+1]=q^3-q^2.\]
\end{enumerate}

\item\label{C5} $\rank\Xi=4$ and $\rank\Xi_{\infty}=2$;
\begin{enumerate}[\mbox{(C\ref{C5}.}1)]
\renewcommand{\theenumi}{\relax}
\item\label{C5.1}$\Xi$ is a cone projecting a hyperbolic quadric and
$\Xi_{\infty}$ is the union of two real planes. Then,
 \[ N=q(q+1)^2+1-(2q^2+q+1)=q^3.\]
\item\label{C5.2} $\Xi$ is a cone projecting an elliptic quadric and and
$\Xi_{\infty}$  the union of two planes defined over $\GF(q^2)$. Then,
\[N=q(q^2+1)+1-(q+1)=q^3. \]
\end{enumerate}

\item\label{C6} $\rank \Xi=\rank \Xi_{\infty}=3$;

 $\Xi$  is the join of a line to a conic and  $\Xi_{\infty}$ is a cone comprising the join of a point to a conic. Then,
 \[N= q^3+q^2+q+1-(q^2+q+1)=q^3.\]

\item\label{C7} $\rank \Xi=3$, $\rank \Xi_{\infty}=2$;
\begin{enumerate}[\mbox{(C\ref{C7}.}1)]
\renewcommand{\theenumi}{\relax}
\item\label{C7.1} $\Xi$ is the join of a line to a conic whereas $\Xi_{\infty}$ is a pair of planes over $\GF(q)$. Then,
\[N=q^3+q^2+q+1-(2q^2+q+1)=q^3-q^2.\]
\item\label{C7.2} $\Xi$ is the join of a line to a conic whereas $\Xi_{\infty}$ is a line. Then,
\[N=q^3+q^2+q+1-q-1=q^3+q^2.\]
\end{enumerate}

\item\label{C8} $\rank \Xi=\rank \Xi_{\infty}=2$;
\begin{enumerate}[\mbox{(C\ref{C8}.}1)]
\renewcommand{\theenumi}{\relax}
\item\label{C8.1}
  $\Xi$ is a pair of solids and $\Xi_{\infty}$ is a pair of planes over $\GF(q)$. Then,
  \[N=2q^3+q^2+q+1-(2q^2+q+1)=2q^3-q^2.\]
\item\label{C8.2} $\Xi$ is a plane and
  $\Xi_{\infty}$ is a line. Then,
  \[N=q^2+q+1-(q+1)=q^2.\]
\end{enumerate}
\end{enumerate}
\end{proof}

Now we are going to use  the same group theoretical
arguments as in \cite[Lemma 2.3]{AG}
in order to be able to fix
the values of some of the parameters in~\eqref{eqeven1} without
losing in generality.
\begin{lemma}\label{main4}
  If $\cQ$ is a hyperbolic quadric, we can assume without
  loss of generality:  \begin{enumerate}
  \item $b=0$, and $a^{q+1}\neq c^{q+1}$
    when $\cC_{\infty}$ is just the point $P_\infty$;
  \item $b=0$, $a=c$ when $\cC_{\infty}$ is a line;
  \item $b=\beta a$, $c=(\beta+1)a$,
    $a\neq 0$ and $\beta^{q+1}=1$, with $\beta\neq 1$ when $\cC_{\infty}$ is the union of
    two lines.
  \end{enumerate}
  If $\cQ$ is a cone, we can assume without loss of generality:
  \begin{enumerate}
    \item $b=0$ when $\cC_{\infty}$ is a point;
    \item $a=b$ when $\cC_{\infty}$ is a line.
    \end{enumerate}
\end{lemma}
\begin{proof}
  Let $\Lambda$ be the set of all lines of $\Sigma_{\infty}$ through
  $P_{\infty}$.
  The action of the stabilizer $G$ of $P_{\infty}$ in $\PGU(4,q)$
  on $\Lambda$
  is the same as the action
  of $\PGU(2,q)$ on the points of $\PG(1,q^2)$.
  This can be easily seen by considering the action on $\PGU(2,q)$
  on the line $\ell$ spanned by $(0,1,0,0)$ and $(0,0,1,0)$ fixing the
  equation $X^{q+1}+Y^{q+1}=0$. Indeed, if $M$ is a $2\times 2$
  matrix representing any $\sigma\in\PGU(2,q)$, then
  $M':=\begin{pmatrix}
    1 & 0 & 0 \\
    0 & M_{\sigma} & 0 \\
    0 & 0 & 1
  \end{pmatrix}$
  represents an element of $\PGU(4,q)$ fixing $P_{\infty}=(0,0,0,1)$.
    The action of $\PGU(2,q)$ on $\ell$ is analyzed in detail in \cite[\S 42]{S}.
  So, we see that the group $G$ has
  two orbits on $\Lambda$, say $\Lambda_1$ and
  $\Lambda_2$ where $\Lambda_1$ consists of the
  totally isotropic lines of $\cH$ through $P_{\infty}$
  while $\Lambda_2$ contains
  the remaining $q^2-q$ lines of $\Sigma_{\infty}$ through $P_{\infty}$.
  Furthermore, $G$ is doubly transitive on
  $\Lambda_1$ and the stabilizer of any $m\in\Lambda_1$ is
  transitive on $\Lambda_2$.

  Let now $\cQ_{\infty}=\cQ\cap\Sigma_{\infty}$.
  If $\cQ$ is hyperbolic and $\cC_{\infty}=\{P_{\infty}\}$ we can assume
  $\cQ_{\infty}$ to be the union of the line $\ell:J=X=0$ and
  another line, say $u:J=aX+cY=0$ with $a^{q+1}\neq c^{q+1}$. Thus, $b=0$.

  Otherwise,
  up to the choice of a suitable element $\sigma\in G$,
  we can always take
  $\cQ_{\infty}$ as the union of
  any two lines in $\{\ell,s,t \}$
  where
  \[ \ell\!:\, J=X=0, \qquad s\!:\, J=X+Y=0,\qquad
  t\!:\,J= X+\beta Y=0
  \]
  with $\beta^{q+1}=1$ and $\beta\neq 1$.

  Actually,  when $\cC_{\infty}$ contains just
  one line we take $\cQ_{\infty}: X(X+Y)=0$,
  while if $\cC_{\infty}$ is the union of two lines we have
  $\cQ_{\infty}: (X+Y)(X+\beta Y)=0$.
  When $\cQ$ is a cone, we get either $\cQ_{\infty}: X^2=0$  or $\cQ_{\infty}: (X+Y)^2=0$.
  The lemma follows.
  \end{proof}

In the next three lemmas  we
denote by $\Xi$ the quadric of $\PG(4,q)$ of  equation \eqref{eqeven1},  whereas by $\Xi_{\infty}$ its section at infinity that is,  the quadric of $\PG(3,q)$ of equation \eqref{eq00}.
\begin{lemma}
\label{2:5}
Suppose $\cQ$ to be a hyperbolic quadric with $\cC_{\infty}$  the union
  of two lines. If
$\rank \Xi_{\infty}=2$, then $\Xi_{\infty}=\Pi_1\cup\Pi_2$ is a plane pair over $\GF(q)$.
\end{lemma}
\begin{proof}
By Lemma \ref{main4} we can assume that $b=\beta a$, $c=(\beta+1)a$, $a\neq 0$ and $\beta^{q+1}=1$ with $\beta\neq 1$.

Furthermore, since $\rank \Xi_{\infty}=2$, we can write
\begin{equation}
\label{eq8}
f(x_0,x_1,y_0,y_1)=(lx_0+mx_1+ny_0+ry_1)(l'x_0+m'x_1+n'y_0+r'y_1);
\end{equation}
for some values of $l,m,n,r$ and $l',m',n',r'$.
Then, the following must be satisfied:
\begin{equation}
\label{eqsys}
\begin{cases}
  ll'=a_1+1\\
  lm'+l'm=1 \\
  l'n+ln'=c_1 \\
  l'r+lr'=c_0+c_1 \\
  mm'=a_0+(1+\nu)a_1+\nu\\
  mn'+nm'=c_0+c_1\\
  mr'+rm'=c_0+(1+\nu)c_1\\
  nr'+rn'=1\\
  nn'=b_1+1\\
  rr'=b_0+(1+\nu)b_1+\nu.
\end{cases}
\end{equation}

If $c_1=0$, then $c_0=1$ as $c^{q+1}=1$; in particular, as $c=b+a$, we have
$a_0+b_0=c_0=1$ and, consequently, as \eqref{eq6} holds  we get $a_1=1=b_1$.
 System \eqref{eqsys} becomes
\[
\begin{cases}
  ll'=0\\
  lm'+l'm=1 \\
  l'n+ln'=0 \\
  l'r+lr'=1 \\
  mm'=a_0+1\\
  mn'+nm'=1\\
  mr'+rm'=1\\
  nr'+rn'=1\\
  nn'=0\\
  rr'=b_0+1.
\end{cases}
\]
So, either $l=0$ or $l'=0$. Suppose the former; then $l'\neq0$ and
$m={l'\,}^{-1}$. We also have $n=0$ and consequently $r={l'\,}^{-1}$.
It follows that
$\Pi_1$ is the plane of equation $x_1+y_1=0$. In particular, both $\Pi_1$
and, consequently, $\Pi_2$ are defined over $\GF(q)$. If $l'=0$, an
analogous argument leads to $\Pi_2:x_1+y_1=0$ and, once more, $\Xi_{\infty}$
splits into two planes defined over $\GF(q)$.

Now suppose
$c_1\neq 0$. From \eqref{eqsys} we get in particular
\begin{equation}\label{eqprime}\begin{cases}
    ll'=a_1+1 \\
    ln'+l'n=c_1 \\
    nn'=b_1+1 \\
    mm'=a_0+(1+\nu)a_1+\nu \\
    mr'+rm'=c_0+(1+\nu)c_1 \\
    rr'=b_0+(1+\nu)b_1+\nu \\
    nr'+rn'=1 \\
    lm'+l'm=1.
  \end{cases}\end{equation}
We obtain $ll'+nn'=a_1+b_1=c_1=ln'+l'n$ and
$mm'+rr'=c_0+(1+\nu)c_1=mr'+rm'$.
Hence,
\[ (l'+n')(l+n)=0,\qquad (m+r)(m'+r')=0. \]
There are the following cases to consider, namely
\begin{enumerate}
\item  $l=n$ and $m=r$;
\item  $l'=n'$ and $m'=r'$;
\item  $l=n$ and $m'=r'$;
\item  $l'=n'$ and $m=r$.
\end{enumerate}
Suppose first $l=n$ and $m=r$; then
$n(n'+l')=c_1\neq 0$; consequently, $n=l\neq 0$ and also $n'\neq l'$.
If $m=0$, then $\Pi_1$ has equation $x_0+x_1=0$ and is defined over
$\GF(q)$; then also $\Pi_2$ is defined over $\GF(q)$ and we are done.

Suppose now $m\neq 0$ (and hence $r\neq 0$). We claim $l/m\in\GF(q)$. This would give that
$\Pi_1$ is defined over $\GF(q)$, whence the thesis.
From~\eqref{eqprime} we have
\[ \begin{cases}
    n'=\frac{b_1+1}{n} \\
    r'=\frac{b_0+(1+\nu)b_1+\nu}{r} \\
    nr'+rn'=1.
    \end{cases} \]
Replacing the values of $n'$ and $r'$ in the last equation we obtain
\begin{equation}
\label{eqA}
 n^2\left(b_0+(1+\nu)b_1+\nu\right)+r^2\left({b_1+1}\right)+nr=0;
 \end{equation}
if we consider
\[ \begin{cases}
    ll'=\frac{a_1+1}{n} \\
    lm'+lm'=1\\
    mm'=a_0+(1+\nu)a_1+\nu
    \end{cases} \]
a similar argument on $l,l',m,m'$ gives
\begin{equation}
l^2\left(a_0+(1+\nu)a_1+\nu\right)+m^2(a_1+1)+lm=0.
\end{equation}
Since, by assumption, $l=n$ and $m=r$ we get
\[ l^2\left(a_0+b_0+(1+\nu)(a_1+b_1)\right)+m^2(a_1+b_1)=0, \]
whence $l^2/m^2\in\GF(q)$. As $q$ is even, this gives $l/m\in\GF(q)$.
The case $l'=n'$ and $m'=r'$ is analogous.

Now suppose $l=n$ and $m'=r'$.
As $l'=\frac{(a_1+1)}{l}$ and $n'+l'=\frac{c_1}{l}$,
\[ n'=\frac{a_1+c_1+1}{l}=\frac{b_1+1}{l}. \]
If $m'=r'=0$, then $\Pi_2$ has equation $(a_1+1)x_0+(b_1+1)y_0=0$ and, consequently, is
defined over $\GF(q)$.
Suppose then $m'=r'\neq 0$. There are several subcases to consider:
\begin{itemize}
\item
if $m=0$, then $m'=r'=l^{-1}$ and $\Pi_2$ has equation
$(a_1+1)x_0+x_1+(b_1+1)y_0+y_1=0$, which is defined over $\GF(q)$;
\item
if $r=0$, then $m'=r'=n^{-1}=l^{-1}$ and we deduce, as above, that $\Pi_2$
is defined over $\GF(q)$;
\item
finally, suppose $m\neq 0\neq r$; then $b_0+(1+\nu)b_1+\nu\neq 0$ and from \eqref{eqsys} we get
\[ m'=\frac{a_0+(1+\nu)a_1+\nu}{m},\qquad
   r'=\frac{b_0+(1+\nu)b_1+\nu}{r}. \]
Since $m'=r'$ we deduce
\begin{equation}\label{m/r} \frac{m}{r}=\frac{a_0+(1+\nu)a_1+\nu}{b_0+(1+\nu)b_1+\nu}\in\GF(q).
\end{equation}
Observe that $l'=(a_1+1)l^{-1}$ and also $r'=(b_0+(1+\nu)b_1+\nu)r^{-1}$; thus from \eqref{eqsys} we obtain
\begin{equation}\label{fi}
l^2(b_0+(1+\nu)b_1+\nu)+r^2(a_1+1)+(c_0+c_1)lr=0.
\end{equation}
On the other hand, since $lm'+l'm=1$,
\[ \frac{l}{m}(a_0+(1+\nu)a_1+\nu)+\frac{m}{l}(a_1+1)=1; \]
using~\eqref{m/r} we obtain
\[ \frac{l}{r}(b_0+(1+\nu)b_1+\nu)+\frac{r}{l}(a_1+1)\left(\frac{a_0+(1+\nu)a_1+1}{b_0+(1+\nu)b_1+1}\right)=1, \]
whence
\begin{equation}
\label{e16}
l^2(b_0+(1+\nu)b_1+\nu)+lr+r^2(a_1+1)\left(\frac{a_0+(1+\nu)a_1+1}{b_0+(1+\nu)b_1+1}\right)=0;
\end{equation}
thus, adding \eqref{fi} to  \eqref{e16} we get
\[ \frac{l}{r}=\frac{a_1+1}{c_0+c_1+1}\left(\frac{a_0+(1+\nu)a_1+1}{b_0+(1+\nu)b_1+1}\right)\in\GF(q) \]
and the plane $\Pi_1$ is defined over $\GF(q)$.
The case $l'=n'$ and $m=r$ is analogous.
\end{itemize}
\end{proof}

\begin{lemma}
\label{lc}
Suppose $\cQ$ to be a hyperbolic quadric a $\cC_{\infty}=\{P_{\infty}\}$. If
$\rank \Xi_{\infty}=2$, then $\Xi_{\infty}$ is a line.
\end{lemma}
\begin{proof}
By Lemma \ref{main4} we can assume $b=0$. Since $\rank \Xi_{\infty}=2$ we have $c^{q+1}=1$.
If $c_1=0$ then
by \eqref{eq6}  we get $a=0$.
In the case in which $c_1\neq 0$ from \eqref{eq7} we again obtain $a=0$.
We have to show that $\Xi_{\infty}$ is the union of two conjugate planes.
In order to obtain this result, it suffices to prove  that   the coefficients
$l, m,n, r$ in \eqref{eq8}  belong to some extension of $\GF(q)$ but
are not in $\GF(q)$.
From \eqref{eqsys} we have
\[
\begin{cases}
  ll'=1\\
  lm'+l'm=1 \\
  mm'=\nu;
 \end{cases}
\]
thus, $\frac{l\nu}{m}+\frac{m}{l}=1$; hence $\nu l^2+l m +m^2=0$.
Since $\Tr_{q}(\nu)=1$ this implies that $\frac{l}{m} \notin\GF(q)$.
\end{proof}

Recall that a quadric $\cQ$
meeting a Hermitian surface $\cH$ in at least $3$ lines
of a regulus is permutable with $\cH$; see
\cite[\S 19.3, pag. 124]{FPS3D}. In particular, in this case the tangent planes
at each of the common points of $\cQ$ and $\cH$ coincide,
that is to say $\cQ$ and $\cH$ intersect nowhere transversally.
We have the following statement.

\begin{lemma} \label{hard}
Suppose $\cQ$ to be hyperbolic and  $\cC_{\infty}$ to
be the union of two lines. If  $\Xi_{\infty}$ is degenerate
then $\Xi$ cannot be a cone projecting a hyperbolic quadric.
\end{lemma}
\begin{proof}
 Suppose $\Xi$ to be
a cone projecting a hyperbolic quadric; then, by Lemma \ref{main2},
Case (C\ref{C5}.\ref{C5.1}),
$|\cH\cap\cQ|=q^3+3q^2+1$.
Let $\cR$ be a regulus of $\cQ$ and denote by $r_1,r_2,r_3$ respectively
the numbers of $1$--tangents, $(q+1)$--secants and $(q^2+1)$-secants to $\cH$ in $\cR$.
A direct counting gives
\begin{equation}\label{r3} \begin{cases}
 r_1+r_2+r_3=q^2+1 \\
 r_1+(q+1)r_2+r_3(q^2+1)=q^3+3q^2+1.
\end{cases}\end{equation}
By straightforward algebraic manipulations we obtain
\[ qr_1+(q-1)r_2=q(q^2-q-1). \]
In particular, $r_2=qt$ with $t\leq q$.
If it were $t=q$, then $r_1=-1$ --- a contradiction;
so $r_2\leq q(q-1)$.
Solving~\eqref{r3} in $r_1$ and $r_3$, we obtain
\[ r_3=\frac{q^2+2q-r_2}{q}\geq\frac{q^2+2q-q^2+q}{q}=3. \]
In particular,
there are at least $3$ lines of $\cR$ contained in $\cH$.
This means that $\cQ$ is  permutable with $\cH$,
see \cite[\S 19.3, pag. 124]{FPS3D} or \cite[\S 86, pag. 154]{S} and
$|\cQ\cap\cH|=2q^3+q^2+1$, a contradiction.
\end{proof}

\section{Proof of Theorem~\ref{main1}}

We use the same setup as in the previous section.
Lemma~\ref{main2}  gives the possible values of $N$ that is, the intersection sizes of the quadric
$\cQ$ and the non-singular Hermitian variety $\cH$  in $\AG(3,q^2)$. Recall that  we have computed $N$ as the difference
$|\Xi|-|\Xi_{\infty}|$ where $\Xi$ is the quadric of $\PG(4,q)$ of  equation \eqref{eqeven1},  whereas $\Xi_{\infty}$ is its section at infinity.

In order to complete the proof of Theorem~\ref{main1} we need to determine the nature of  $\cC_{\infty}=\{\cQ\cap\cH\} \setminus \AG(3,q^2)$ in all possible cases.
\subsection{The elliptic case}
Let $\cQ$ be an elliptic quadric. In this case $\cC_{\infty}=\{P_{\infty}\}$.
The possible sizes for the affine part of
$\cH\cap\cQ$
correspond to cases (C\ref{C1}), (C\ref{C2}), (C\ref{C3}), (C\ref{C4}) and (C\ref{C6}) of Lemma \ref{main2}, whence
\[|\cQ\cap\cH|=N+1 \in \{\begin{array}[t]{l}q^3-q^2+1, q^3-q^2+q+1, q^3-q+1, q^3+1, q^3+q+1,\\ q^3+q^2-q+1, q^3+q^2+1 \}. \end{array} \]

\subsection{The degenerate case}

Suppose $\cQ$ to be a cone.
Since $c=0$ in~\eqref{eqQ},  we get that $\Xi_{\infty}$ is non degenerate as $\det A_{\infty}=1$; hence,
the size of the affine part of $\cH\cap\cQ$
falls in one of cases (C\ref{C1}) or (C\ref{C3}) in Lemma \ref{main2}.
Here $\cC_{\infty}$ is
either a point or one line.

If  $\cC_{\infty}$ consists of just $1$  point, then we can assume $b=0$  in \eqref{eqQ} by Lemma \ref{main4}.
Thus \eqref{alfa} becomes $\alpha=a_0+a_1+(1+\nu)a_1^2+a_0a_1$ and hence
 cases (C\ref{C1}) and (C\ref{C3}) in Lemma \ref{main2} may happen; so,
\[|\cQ\cap\cH|=N+1 \in \{q^3-q^2+q+1, q^3-q+1, q^3+q+1, q^3+q^2-q+1 \}. \]

If  $\cC_{\infty}$ consists of $q^2+1$ points on a
line, then
we can assume $a=b$  in \eqref{eqQ} by Lemma \ref{main4}.
In this case $\alpha=0$; thus,  only subcases  (C\ref{C1}.\ref{C1.1}) of (C\ref{C1}) and (C\ref{C3}.\ref{C3.1}) of (C\ref{C3})  in Lemma \ref{main2} may occur.
In particular, \[|\cQ\cap\cH|=N+q^2+1 \in \{q^3+q^2-q+1, q^3+2q^2-q+1 \}. \]

\subsection{The hyperbolic case}

If $\cQ$ is a hyperbolic quadric, then we have three possibilities
for $\cC_{\infty}$, that is $\cC_{\infty}$ is either a point, or a line or the union of two lines.

If $\cC_{\infty}=\{P_{\infty}\}$, then all  cases (C\ref{C1})--(C\ref{C8}) of Lemma \ref{main2}
might occur. We are going to show that some subcases of Lemma \ref{main2} can be excluded.

When $\rank \Xi_{\infty}=2$, from Lemma
\ref{lc} we have that  subcases (C\ref{C7}.\ref{C7.2})   and (C\ref{C8}.\ref{C8.1}) cannot occur.
So,
\[|\cQ\cap\cH|=N+1 \in\{\begin{array}[t]{l} q^2+1, q^3-q^2+1, q^3-q^2+q+1, q^3-q+1,  q^3+1,  q^3+q+1, \\  q^3+q^2-q+1, q^3+q^2+1 \}.\end{array}\]

Suppose now $\cC_{\infty}$ to be exactly one line.  By Lemma~\ref{main4}
we can assume $b=0$ and $a=c$  in~\eqref{eqQ}.

 When  $\Xi_{\infty}$ is non degenerate,  only cases  (C\ref{C1}) and (C\ref{C3}) may occur.  Observe that \eqref{alfa} becomes

\[ \alpha=\frac{c_0+c_1}{\gamma}+
   \frac{1}{\gamma^2}\big[(1+\nu)(c_1^2)+c_0c_1 \big]. \]

 Since $\gamma=c_0^2+\nu c_1^2+c_0c_1+1$ we have  $c_1^2+\nu c_1^2+c_1c_0=\gamma+c_0^2+c_1^2+1$. Therefore
\[ \alpha=\frac{(c_0+c_1)}{\gamma}+\frac{c_0^2+c_1^2}{\gamma^2}+\frac{1}{\gamma}+\frac{1}{\gamma^2} \]
and
$\Tr_{q}(\alpha)=0$.
 Hence,  subcases  (C\ref{C1}.\ref{C1.2}) and (C\ref{C3}.\ref{C3.2}) of Lemma \ref{main2} cannot  happen;
so,
\[ |\cQ\cap\cH|=N+q^2+1 \in\{ q^3+q^2-q+1, q^3+2q^2-q+1 \}. \]

Assume now that $\Xi_{\infty}$ is degenerate. Cases (C\ref{C2}) and (C\ref{C4})--(C\ref{C8}) of Lemma \ref{main2} occur.
 We are going to show that
$\rank \Xi_{\infty}=3$.
Suppose, on the contrary, $\rank \Xi_{\infty}=2$.
 As we have seen in the proof of Lemma \ref{main2} if it were $c_1=0$, from \eqref{eq6} we would
 have $a_1=a_0=0$, that is $a=0$, which is impossible. So, $c_1\neq 0$;
since $b_1=b_0=0$,
we would now have from \eqref{eq7} $a_1=a_0=0$---again  a contradiction.

Thus, only cases (C\ref{C2}) and (C\ref{C4})
 in Lemma \ref{main2} might happen;  in particular,
\[ |\cQ\cap\cH|=N+q^2+1\in\{q^3+1, q^3+q^2+1, q^3+2q^2+1 \}. \]

Finally, when $\cC_{\infty}$ consists of two lines,  by Lemma~\ref{main4}  we can assume
$b=\beta a$, $c=(\beta+1)a$ where $a\neq 0$ and $\beta^{q+1}=1$  in \eqref{eqQ}.

Suppose now $\Xi_{\infty}$ to be non degenerate. Cases (C\ref{C1}) and (C\ref{C3}) of Lemma \ref{main2} occur.

From $c=a+b$ we get $c^{q+1}=a^{q+1}+a^qb+b^qa+b^{q+1}$. Since $b^{q+1}=a^{q+1}$, we have
$c^{q+1}=a^qb+b^qa$, that is
  \[a_0b_1+a_1b_0=c_0^2+\nu c_1^2+c_0c_1,\]

 On the other hand,  from $c_0=a_0+b_0$ and $c_1=a_1+b_1$ we obtain
  $c_0c_1=a_0a_1+a_0b_1+a_1b_0+b_0b_1$
 that is \[  a_0a_1+b_0b_1=c_0^2+\nu c_1^2. \]
 Now $(a_0b_1+a_1b_0)( a_0a_1+b_0b_1)= a_0b_0(a_1^2+b_1^2)+a_1b_1(a_0^2+b_0^2)=a_0b_0c_1^2+a_1b_1c_0^2=(c_0^2+\nu c_1^2+c_0c_1)(c_0^2+\nu c_1^2)$ and thus  \eqref{alfa} becomes
 \[\alpha = \frac{c_0+c_1}{1+c_0^2+\nu c_1^2+c_0c_1}+ \frac{(c_0+c_1)^2}{(1+c_0^2+\nu c_1^2+c_0c_1)^2};\]
so, $\alpha$ has trace $0$.

Hence just  subcases (C\ref{C1}.\ref{C1.1}) and (C\ref{C3}.\ref{C3.1})
in
Lemma \ref{main2} may occur and
\[ |\cQ\cap\cH|=N+2q^2+1\in\{q^3+2q^2-q+1,q^3+3q^2-q+1\}. \]

If $\Xi_{\infty}$ is degenerate, then cases (C\ref{C2}) and  (C\ref{C4})--(C\ref{C8}) in Lemma \ref{main2} may happen.
By Lemma  \ref{hard}  only subcases
(C\ref{C4}.\ref{C4.2}) in (C\ref{C4}) and (C\ref{C5}.\ref{C5.2}) in (C\ref{C5}) are possible.

In particular when $\rank \Xi_{\infty}=2$,
it follows from Lemma \ref{2:5} that  only subcases
(C\ref{C7}.\ref{C7.1}) in (C\ref{C7})  and (C\ref{C8}.\ref{C8.1})  in (C\ref{C8})  may occur.
Thus we get
\[ |\cQ\cap\cH|=N+2q^2+1 \in\{ q^3+q^2+1, q^3+2q^2+1, 2q^3+q^2+1 \} \]
 and the proof is completed.

It is straightforward to see, by means of a computer aided computation for small values of $q$ that
all the cardinalities enumerated above may occur.


\section{Extremal configurations}
\label{EC}
As in the case of odd characteristic, it is possible to provide
a geometric description of the intersection configuration when
the size is either $q^2+1$ or $2q^3+q^2+1$.
These values are
respectively the minimum and the maximum yielded by Theorem \ref{main1}.
and they can happen only when $\cQ$ is an hyperbolic quadric.
Throughout this section we assume that the hypotheses of
Theorem \ref{main1} hold, namely that $\cH$ and $\cQ$ share a tangent
plane at some point $P$.

\begin{theorem}
\label{ppt1}
Suppose $|\cH\cap\cQ|=q^2+1$. Then, $\cQ$ is a hyperbolic quadric
and $\Omega=\cH\cap\cQ$ is an ovoid of $\cQ$.
\end{theorem}
\begin{proof}
  By Theorem \ref{main1}, $\cQ$ is hyperbolic.
  Fix a regulus $\cR$ on $\cQ$. The  $q^2+1$ generators of $\cQ$ in $\cR$ are pairwise disjoint
  and each has non-empty intersection with $\cH$; so there can be at most
  one point of $\cH$ on each of them. It follows that $\cH\cap\cQ$ is an
  ovoid.
   In particular,
   by the above argument, any generat of $\cQ$ through a point of
   $\Omega$ must be tangent to $\cH$. Thus, at all points of $\Omega$
   the tangent planes   to $\cH$ and to $\cQ$  are the same.
\end{proof}

\begin{theorem}
\label{ppt2}
Suppose $|\cH\cap\cQ|=2q^3+q^2+1$. Then, $\cQ$ is a hyperbolic quadric
permutable with $\cH$.
\end{theorem}
Theorem \ref{ppt2} can be obtained as a consequence of
the analysis contained in \cite[\S 5.2.1]{E1},
in light of \cite[Lemma 19.3.1]{FPS3D}.


\begin{thebibliography}{999}
\bibitem{AG} A. Aguglia, L. Giuzzi, {Intersection of the Hermitian surface
with Irreducible quadrics in $\PG(3,q^2)$, $q$ odd}, \emph{Finite Fields Appl.},
{\bf 30} (2014), 1-13.
\bibitem{BBFS} D. Bartoli, M. De Boeck, S. Fanali, L. Storme,
On the functional codes defined by quadrics and Hermitian varieties,
\emph{Des. Codes Cryptogr.} {\bfseries 71} (2014), 21-46.
\bibitem{IG1} I. Cardinali, L. Giuzzi, {Codes and caps from orthogonal
    Grassmannians}, \emph{Finite Fields Appl.}, {\bf 24} (2013), 148-169.
\bibitem{IG2} I. Cardinali, L. Giuzzi, K. Kaipa, A. Pasini,
  {Line Polar Grassmann Codes of Orthogonal Type}, \emph{J. Pure Applied Algebra}
  {\bf 220} (2016), 1924-1934.
\bibitem{IG3} I. Cardinali, L. Giuzzi,
  {Minimum Distance of Symplectic Grassmann Codes}, \emph{Linear Algebra Appl.}
  {\bf 488} (2016), 124-134.
\bibitem{CP} A. Cossidente, F. Pavese, On the intersection of a Hermitian surface
  with an elliptic quadric, \emph{Adv. Geom.}, {\bf 15} (2015), 233-239.
\bibitem{E1} F.A.B. Edoukou, Codes defined by forms of degree $2$ on Hermitian
surfaces and S{\o}rensen's conjecture, \emph{Finite Fields Appl.} {\bfseries 13} (2007) 616-627.
\bibitem{EHRS} F.A.B. Edoukou, A. Hallez, F. Rodier, L. Storme, The small weight codewords of the functional codes associated to non-singular Hermitian varieties.  \emph{Des. Codes Cryptogr.} {\bf 56} (2010), no. 2-3, 219-233.
 \bibitem{E2} F.A.B. Edoukou, Codes defined by forms of degree 2 on non-degenerate Hermitian varieties in $P^4(F_q)$. \emph{Des. Codes Cryptogr.} {\bf 50} (2009), no. 1, 135-146.
\bibitem{E3} F.A.B. Edoukou, Structure of functional codes defined on non-degenerate Hermitian varieties, \emph{J. Combin Theory Series A}, {\bfseries 118}
(2010), 2436-2444.
\bibitem{EH} D. Eisenbud, J. Harris, \emph{3264 and all that: Intersection theory in algebraic geometry}, manuscript (2013).
\bibitem{fulton} W. Fulton, \emph{Intersection theory}, Springer Verlag, (1998).
\bibitem{FPS3D} J.W.P. Hirschfeld, \emph{Finite Projective Spaces of Three Dimensions}, Oxford University Press (1985).
\bibitem{HT} J.W.P. Hirschfeld, J. A.  Thas, {\em General Galois Geometries}, Springer Verlag (2015).
 \bibitem{PGOFF} J.W.P. Hirschfeld, \emph{Projective Geometries over Finite Fields},2nd Edition, Oxford University Press,New York, (1998).
\bibitem{HTV} J.W.P. Hirschfeld, M. Tsfasman, S. Vl\v{a}du\c{t},
  The weight hierarchy of higher-dimensional Hermitian codes,
\emph{IEEE Trans. Inform. Theory} {\bf 40} (1994), 275-278.
\bibitem{HS} A. Hallez, L. Storme, Functional codes arising from quadric intersections with Hermitian varieties, \emph{Finite Fields Appl.}  {\bfseries 16}
(2010) 27-35.
\bibitem{L} G. Lachaud, Number of points of plane sections and linear
codes defined on algebraic varieties, \emph{Arithmetic, Geometry and
Coding Theory; Luminy, France, 1993}, Walter de Gruyter (1996), 77-104.
 \bibitem{S} B. Segre, \emph{Forme e geometrie Hermitiane con particolare
 riguardo al caso finito}, Ann. Mat. Pura Appl. (4) {\bfseries 70} (1965).
 \bibitem{tvn} M. Tsfasman, S. Vl\v{a}du\c{t}, D. Nogin,
   \emph{Algebraic Geometry Codes: Basic Notions},
   Mathematical Surveys and Monographs 139, A.M.S. (2007).
\bibitem{tv} M. Tsfasman, S. Vl\v{a}du\c{t}, {Geometric approach to
    higher weights}, \emph{IEEE Trans. Inform. Theory} {\bfseries 41} (1995),
    1564-1588.
\end{thebibliography}
\end{document}